\documentclass[oneside, book,11 pt]{amsart}
\usepackage{color}
\usepackage{cancel}
\usepackage{amsmath}
\usepackage{amssymb}
\usepackage[normalem]{ulem}
\usepackage{amsfonts}
\usepackage{bbm}
\usepackage{wrapfig}
\usepackage{subcaption}
\usepackage{framed}
\usepackage{tcolorbox}
\usepackage{adjustbox}
\usepackage{amsthm}
\usepackage{amsthm}
\usepackage{enumerate}
\usepackage[mathscr]{eucal}
\usepackage{graphicx}
\usepackage{pstricks}
\usepackage[active]{srcltx}
\usepackage{fancyhdr}
\usepackage{verbatim}
\usepackage{fullpage}
\usepackage{hyperref}
\usepackage{mathtools}
\usepackage{enumitem,kantlipsum}
\usepackage{epstopdf}

\hypersetup{
colorlinks = true, 
urlcolor = blue, 
linkcolor = red, 
citecolor = blue 
}
\usepackage{caption}
\usepackage{float}

\newtheorem{thm}{Theorem}[section]
\newtheorem*{thm*}{Theorem}
\newtheorem{lemma}[thm]{Lemma}
\newtheorem{corollary}[thm]{Corollary}

\newtheorem{prop}[thm]{Proposition}
\newtheorem*{prop*}{Proposition}

\newtheorem{remark}[thm]{Remark}
\numberwithin{equation}{section}

\def\N{\mathbb N}

\def\br{\mathbf r}

\newcommand{\Z}{\mathbb{Z}}

\def \E{\mathbb E}
\def \R{\mathbb R}

\def \E1{\mathcal E}

\def \T{\mathcal T}
\def\T1{\mathbb T}

\def\ga{\gamma}
\def\t{\widetilde}
\def \m{\vec}

\def\mi{{\vec{i}}}
\def\mj{{\vec{j}}}

\def\mt{\textbf{m}}
\def\tg{g}
\def\rT{\textrm{T}}
\def\emp{\varnothing}

\def\.{\hskip.06cm}
\def\ts{\hskip.03cm}

\def\nin{\noindent}

\def\Hom{\textup{\textrm{Hom}}}
\def\Cull{\textup{\textrm{Ext}}}
\def\St{\textup{\textrm{St}}}
\def\NP{\textup{\textsf{NP}}}
\def\OO{O}

\title{Kirszbraun-Type Theorems for Graphs}
\author{
Nishant Chandgotia
\and
Igor Pak
\and
Martin Tassy
}\address{School of Mathematical Sciences\\
Tel Aviv University,
Israel}
\email {nishant.chandgotia@gmail.com}

\address{Mathematics Department\\
Unversity of California, Los Angeles, CA}
\email{pak@math.ucla.edu}

\address{Mathematics Department\\
Dartmouth College, Hanover, NH}
\email{martintassy@gmail.com}

\begin{document}
\maketitle

\begin{abstract} The classical \emph{Kirszbraun theorem} says that all $1$-Lipschitz
functions $f:A\longrightarrow \R^n$, $A\subset \R^n$, with the Euclidean metric have
a $1$-Lipschitz extension to $\R^n$. For metric spaces $X,Y$ we say that $Y$ is
\emph{$X$-Kirszbraun} if all $1$-Lipschitz functions $f:A\longrightarrow Y$,
$A\subset X$, have a $1$-Lipschitz extension to~$X$. We analyze the case when $X$ and $Y$ are
graphs with the usual path metric.  We prove that $\Z^d$-Kirszbraun graphs are exactly
graphs that satisfies a certain \emph{Helly's property}.  We also consider complexity
aspects of these properties.
\end{abstract}

\

\vskip.7cm

\section{Introduction}\label{section:intro}

\nin
Discretizing results in metric geometry is important for many applications,
ranging from discrete differential geometry to numerical methods.  The discrete
results are stronger as they typically imply the continuous results in the limit.
Unfortunately, more often than not, straightforward discretizations fall apart;
new tools and ideas are needed to even formulate these extensions; see e.g.~\cite{BS,Lin}
and~\cite[$\S$21--$\S$24]{Pak}.

In this paper we introduce a new notion of \emph{$G$-Kirszbraun graphs},
where~$G$ is vertex-transitive graph.  The idea is to discretize the classical
\emph{Kirszbraun theorem} in metric geometry~\cite{kirszbraun1934}
(see also~\cite[$\S$1.2]{BL}).  Our main goal
is to explain the variational principle for the height functions of tilings introduced
by the third author in~\cite{Tas} and further developed in~\cite{MR3530972,TassyMenz2016}
(see Section~\ref{section: motivation}); we also aim to lay a proper foundation
for the future work.

Our second goal is to clarify the connection to the \emph{Helly theorem},
a foundational result in convex and discrete geometry~\cite{helly1923}
(see also~\cite{MR0157289,Mat}).  Graphs that satisfy the \emph{Helly's property}
has been intensely studied in recent years~\cite{MR2405677}, and we establish
a connection between two areas.   Roughly, we show that $\Z^d$-Kirszbraun
graphs are somewhat rare, and are exactly the graphs that satisfy the
Helly's property with certain parameters.

\smallskip

\subsection{Main results}
Let $\ell_2$ denote the usual Euclidean metric on $\R^n$ for all $n$. Given a metric
space~$X$ and a subset $A$, we write $A\subset X$ to mean that the subset~$A$ is
endowed with the restricted metric from~$X$. The \emph{Kirszbraun theorem} says
that for all \ts $A\subset (\R^n,\ell_2)$, and all Lipschitz functions \ts
$f: A \longrightarrow (\R^n,\ell_2)$, there is an extension to a
Lipschitz function on $\R^n$ with the same Lipschitz constant.

Recall now the \emph{Helly theorem}: Suppose a collection of convex sets
$B_1, B_2, \ldots, B_k$ satisfies the property that every $(n+1)$-subcollection has a
nonempty intersection, then $\cap B_i\neq \emp$.  Valentine in~\cite{MR0011702}
famously showed how the Helly theorem can be used to obtain the Kirszbraun theorem.
The connection between these two theorems is the key motivation behind this paper.

\smallskip

Given metric spaces $X$ and $Y$, we say that $Y$ is \emph{$X$-Kirszbraun}
if all $A\subset X$, every $1$-Lipschitz maps \ts $f: A \longrightarrow Y$ \ts
has a $1$-Lipschitz extension from~$A$ to~$X$. In this notation, the Kirszbraun
theorem says that \ts $(\R^n,\ell_2)$ \ts is \ts $(\R^n,\ell_2)$-Kirszbraun.

Let $m\in \N$ and $n\in \N\cup\{\infty\}$, $n>m$.  Metric space $X$
is said to have \emph{$(n, m)$-Helly's property} if for every collection of closed balls
$B_1, B_2, \ldots, B_n$ of radius~$\ge~1$ whenever every $m$-subcollection has a
nonempty intersection, we have $\cap_{i=1}^n B_i\neq \emp$.
Since balls in $\R^n$ with the Euclidean metric are convex,
the Helly theorem can be restated to say that $(\R^n,\ell_2)$ is $(\infty, n+1)$-Helly.
Note that the metric is important here, i.e.\
$(\R^n,\ell_\infty)$ is $(\infty, 2)$-Helly, see e.g.~\cite{MR0157289,Pak}.

Given a graph $H$, we endow the set of vertices (also denoted by $H$) with the path metric.
By~$\Z^d$ we mean the Cayley graph of the group $\Z^d$ with respect to standard generators.
All graphs in this paper are nonempty, connected and simple (no loops and multiple edges).
The following is the main result of this paper.

\begin{thm}[Main theorem]\label{thm:Zd Kirszbraun}
Graph $H$ is $\Z^d$-Kirszbraun if and only if $H$ is $(2d, 2)$-Helly.\label{theorem: main result}
\end{thm}

Let $K_n$ denote the \emph{complete graph} on $n$-vertices. Clearly, $K_n$ is $G$-Kirszbraun
for all graphs $G$, since all maps $f: G\longrightarrow K_n$ are $1$-Lipschitz. On
the other hand, $\Z^2$ is not $\Z^2$-Kirszbraun, see Figure~\ref{f:Z2}.  This example
can be modified to satisfy a certain extendability property, see $\S$\ref{ss:ext-bip}.

\begin{figure}[hbt]
 \begin{center}
   \includegraphics[height=4.2cm]{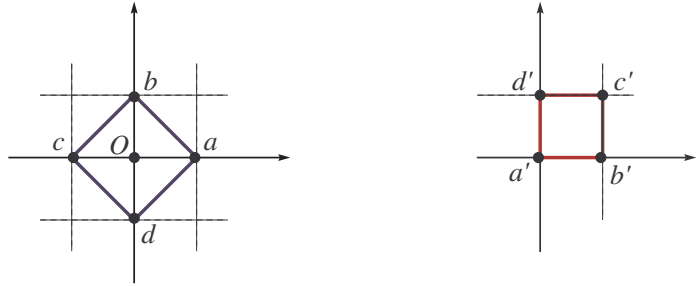}
   \caption{Here $A=\{a,b,c,d\}\subset \Z^2$.  Define $f: a\to a'$,
   $b\to b'$, $c\to c'$, $d\to d'$.  Then $f: A\to \Z^2$ is $1$-Lipschitz
   but not extendable to $\{O,a,b,c,d\}$.}
   \label{f:Z2}
 \end{center}
\end{figure}

\smallskip

\subsection{Structure of the paper}
We begin with a short non-technical Section~\ref{section: motivation}
describing some background results and ideas.  In essence, it is a
remark which is too long to be in the introduction.  It reflects the
authors' different points of view on the subject, which includes
ergodic theory, geometric combinatorics and discrete probability.
While in principle this section it can be skipped, we do
recommend reading it as it motivates other parts of the paper.

We then proceed to prove Theorem~\ref{thm:Zd Kirszbraun} in
Section~\ref{s:proof-main}. In Section~\ref{section: extensions of main},
we present several extensions and applications of the main theorem.  These
section is a mixed bag: we include a continuous analogue of the main
theorem (Theorem~\ref{thm: rd version}), the extension to larger
integral Lipschitz constants (Theorem~\ref{thm: about t lipschitz maps}),
and the bipartite extension useful for domino tilings
(Theorem~\ref{thm: bipartite version of main theorem}).
In a short Section~\ref{section:recognisability and an application},
we discuss computational aspects of the $\Z^d$-Kirszbraun property,
motivated entirely by applications to tilings.  We conclude with final
remarks and open problems in Section~\ref{s:finrem}.

\bigskip

\section{Motivation and background}\label{section: motivation}
A \emph{graph homomorphism} from a graph $G$ to a graph $H$ is an adjacency
preserving map between the respective vertices. Let $\Hom(G, H)$ denote the set of all
graph homomorphisms from $G$ to~$H$. We refer to~\cite{HN} for background on
graph homomorphisms and connections to coloring and complexity problems.

Our motivation comes from two very distinct sources:
\begin{enumerate}
\item
Finding `fast' algorithms to determine whether a given graph homomorphism on the boundary of a box in $\Z^d$ to $H$ extends to the entire box.\label{motivation: 1}
\item
Finding a natural parametrization of the so-called ergodic Gibbs measures on space of graph homomorphisms $\Hom(\Z^d, H)$ (see \cite{MR2251117,TassyMenz2016}).
\label{motivation: 2}
\end{enumerate}

For~$(1)$, roughly, suppose we are given a certain simple set of tiles~$\rT$,
such as dominoes or more generally bars \ts $\{k\times 1, 1 \times \ell\}$.  It turns out, that
$\rT$-tileability of a simply-connected region $\Gamma$ corresponds to existence of a graph
homomorphism with given boundary conditions on $\partial\Gamma$. We refer to~\cite{Pak-horizons,Thu}
for the background, and to~\cite{MR3530972,Tas} for further details. Our
Theorem~\ref{prop: hole filling} is motivated by these problems.

For both these problems, the $\Z^d$-Kirszbraun property of the graph $H$ (or a related graph)
is critical and motivates this line of research; the space of
$1$-Lipschitz maps is the same as the space of graph homomorphisms if and only if $H$ is
\emph{reflexive}, that is, every vertex has a self-loop. The study of
Kirszbraun-type theorems among metric spaces and its relationship to Helly-like properties
is an old one and goes back to the original paper by Kirszbraun~\cite{kirszbraun1934}.
A short and readable proof is given in \cite[p.~201]{federer1969}.
This was later rediscovered in \cite{MR0011702} where it was generalized to the cases
where the domain and the range are spheres in the Euclidean space or Hilbert spaces.
The effort of understanding which metric spaces
satisfy Kirszbraun properties culminated in the theorem by
Lang and Schroeder~\cite{langschroder1997} that identified
the right curvature assumptions on the
underlying spaces for which the theorem holds.

In the metric graph theory, the research has focused largely on a certain universality property.
Formally, a graph is called \emph{Helly} if it is $(\infty, 2)$-Helly. An easy deduction,
for instance following the discussion in \cite[Page 153]{MR0157289}, shows that $H$ is
Helly if and only if for all graphs $G$, $H$ is $G$-Kirszbraun. Some nice
characterizations of Helly graphs can be found in the survey~\cite[$\S$3]{MR2405677}.
However, we are not aware of any other study of $G$-Kirszbraun graphs for fixed~$G$.

\bigskip

\section{Proof of the main theorem}\label{s:proof-main}

\subsection{Geodesic extensions}
Let $d_H$ denote the path metric on the graph $H$.
A \emph{walk}~$\ga$ in the graph $H$ of \emph{length}~$k$,
is a sequence of $k+1$ vertices $(v_0, v_1, \ldots,v_k)$, s.t.
$d_H(v_{i}, v_{i+1})\leq 1$, for all $0\leq i \leq k-1$.
We say that $\ga$ starts at $v_0$ and ends at $v_k$.
A \emph{geodesic} from vertex $v$ to $w$ in a graph~$G$
is a walk~$\ga$ from $v$ to $w$ of the shortest length.

Consider a graph $G$, a subset $A\subset G$ and $b\in G\setminus A$.  Define
the \emph{geodesic extension of $A$ with respect to $b$} \ts as the following set:
$$\aligned
\Cull(A,b)\, := & \, \, \{a\in A~:~\text{there does not exist \ts $a'\in A\setminus \{a\}$ \ts s.t.\ there is a }\\
& \hskip1.86cm \text{geodesic~$\ga$ from $a$ to $b$ which passes through $a'$}\}.
\endaligned
$$
For example, let $A\subset \Z^2\setminus \{(0,0)\}$. If $(i,j), (k,l)\in \Cull(A,\m 0)$
are elements of the same quadrant, then $|i|> |k|$ if and only if $|j|<|l|$.

\begin{remark}\label{remark: zd culling coordinates}{\rm
If $A\subset \Z^d$ be contained in the coordinate axes then $|\Cull(A,\m 0)|\leq 2d$. }
\end{remark}

The notion of geodesic extension allows us to prove that certain $1$-Lipschitz
maps can be extended:

\begin{prop}\label{prop:culling of extras}
Let $A\subset G$, map $f: A\longrightarrow H$ be $1$-Lipschitz, and let $b\in G \setminus A$.
The map $f$ has a $1$-Lipschitz extension to \ts $A\cup \{b\}$ \ts if and only if \ts
$f|_{\Cull(A,b)}$ \ts has a $1$-Lipschitz extension to \ts $\Cull(A,b)\cup\{b\}$.
\end{prop}

\begin{proof}
The forward direction of the proof is immediate because $\Cull(A,b)\subset A$. For the backwards direction let $\t f: \Cull(A,b)\cup\{b\}\subset G\longrightarrow
H$ be a $1$-Lipschitz extension of $f|_{\Cull(A,b)}$ and consider the map
$\hat f: A\cup \{b\}\subset G\longrightarrow H$
given by
$$\hat f(a):=\begin{cases}
f(a)&\text{ if }a\in A\\
\t f(b)&\text{ if }a=b.
\end{cases}•$$
To prove that $\hat f$ is $1$-Lipschitz we need to verify that for all $a\in A$, $d_ H(\hat f(a), \hat f(b))\leq d_ G(a,b)$.
From the hypothesis it follows for $a\in \Cull(A,b)$. Now suppose $a\in A\setminus \Cull(A,b)$. Then there exists $a'\in \Cull(A,b)$ such that there exists a
geodesic from $a$ to $b$ passing through $a'$. This implies that $d_{ G}(a,b)=d_{ G}(a,a')+d_ G(a',b)$. But
\begin{eqnarray*}
d_ G(a,a')\geq d_{ H}(\hat f(a), \hat f (a'))=d_{ H}(f(a), f (a'))\text{ because }f\text{ is $1$-Lipschitz}\\
d_ G(a',b)\geq d_{ H}(\hat f(a'), \hat f (b))=d_{ H}(\t f(a'),\t f (b))\text{ because }\t f\text{ is $1$-Lipschitz}.
\end{eqnarray*}
By the triangle inequality, the proof is complete.
\end{proof}

\smallskip

\subsection{Helly's property}
Given a graph $H$, a vertex $v\in H$ and $n\in \N$ denote by $B^H_n(v)$,
the ball of radius $n$ in $H$ centered at $v$. We will now interpret
the $(n,2)$-Helly's property in a different light.

\begin{prop}\label{prop: small culled vertices}
Let $H$ be a graph satisfying the $(n,2)$-Helly's property.
For all 1-Lipschitz maps $f: A\subset G\longrightarrow
H$ and $b\in G\setminus A$ such that $|\Cull(A, b)|\leq n$,
there exists a $1$-Lipschitz extension of $f$ to $A\cup \{b\}$.
\end{prop}

\begin{proof}	Consider the extension $\t f$ of $f$ to the set $A\cup \{b\}$,
 where $\t f(b)$ is any vertex in
$$
\bigcap _{b'\in \Cull(A, b)}B^H_{d_G(b,b')}(f(b'));
$$
the intersection is nonempty because $|\Cull(A, b)|\leq n$ and for all $a, a'\in \Cull(A, b)$, we have:
$$d_H\bigl(f(a), f(a')\bigr) \, \leq \, d_G(a, a')\, \leq \, d_G(a, b)\. + \. d_{G}(b, a')
$$
which implies
$$B^H_{d_G(a,b)}(f(a))\cap B^H_{d_G(b,a')}\bigl(f(a')\bigr)\, \neq \, \emp.
$$
The function \ts $\t f|_{\Cull(A,b)\cup\{b\}}$ is $1$-Lipschitz, so
Proposition~\ref{prop:culling of extras} completes the proof.
\end{proof}

\subsection{Examples}
Let $C_n$ and $P_n$ denote the \emph{cycle graph} and the \emph{path graph} with $n$ vertices, respectively.

\begin{corollary}
All connected graphs are \ts $P_n$--, \ts $C_n$-- and \ts $\Z$--Kirszbraun.
\end{corollary}

In the case when $G=P_n, C_n$ or $\Z$ we have for all $A\subset G$ and
$b\in G\setminus A$, $|\Cull(A,b)|\leq 2$; the corollary follows from
Proposition~\ref{prop: small culled vertices} and the fact that
all graphs are $(2,2)$-Helly.

Let $\br = (r_1, r_2, \ldots r_n)\in \N^n$.  Denote by $T_{\br}$ the \emph{star-shaped tree}
with a central vertex $b_0$ and $n$ disjoint walks of lengths
$r_1,\ldots,r_n$ emanating from it and ending in vertices $b_1, b_2, \ldots, b_n$.
\begin{corollary}\label{cor: helly and test tree}
Graph $H$ has the $(n,2)$-Helly's property if and only if $H$ is $T_{\br}$-Kirszbraun, for all $\br\in \N^n$.
\end{corollary}
\begin{proof}
For all $\br\in \N^n$, $A\subset T_{\br}$ and $b\in T_{\br}\setminus \{A\}$, we have \ts $|\Cull(A,b)|\leq n$. Thus by Proposition \ref{prop: small culled vertices} if $H$ has the $(n,2)$-Helly's property then $H$ is $T_{\br}$-Kirszbraun. For the other direction, let $H$ be $T_{\br}$-Kirszbraun for all $\br\in \N^n$. Suppose that $$B^H_{r_1}(v_1), B^H_{r_2}(v_2), \ldots, B^H_{r_n}(v_n)$$ are balls in $H$ such that $B^H_{r_i}(v_i)\cap B^H_{r_j}(v_j)\neq \emptyset$ for all $1\leq i, j \leq n$. Then $f: \{b_1, b_2, \ldots, b_n\}\subset T_{\br}\to H$ given by $f(b_i):= v_i$ is $1$-Lipschitz with a $1$-Lipschitz extension $\t f: T_{\br}\to H$. It follows that $\t f (b_0)\in \bigcap_{i=1}^n B^H_{r_i}(v_i)$ proving that $H$ has the $(n,2)$-Helly's property.
\end{proof}
\smallskip

\subsection{Proof of Theorem~\ref{thm:Zd Kirszbraun}}
We will first prove the ``only if'' direction. Let $H$ be a graph which is
$\Z^d$-Kirszbraun. For all $\br\in \N^{2d}$ there is an isometry from $T_{\br}$ to
$\Z^d$ mapping the walks emanating from the central vertex to the coordinate axes.
Hence $H$ is $T_{\br}$-Kirszbraun for all $\br\in \N^{2d}$. By Corollary
~\ref{cor: helly and test tree}, we have proved the $(2d,2)$-Helly's property for~$H$.

In the ``if'' direction, suppose $H$ has the $(2d,2)$-Helly's property.
We need to prove that for all $A\subset \Z^d$, every
$1$-Lipschitz maps $f:A\to H$ has a $1$-Lipschitz extension.
It is sufficient to prove this for finite subsets~$A$.
We proceed by induction on $|A|$.  Namely, we prove the following property $\St(n)$:

\smallskip

\noindent
\qquad  Let $f:A\subset \Z^d\longrightarrow H$ be $1$-Lipschitz with $|A|=n$. Let $b\in \Z^d\setminus A$. Then the function $f$ has

\noindent
\qquad  a $1$-Lipschitz extension to $A\cup \{b\}$.

\smallskip

\noindent
We know $\St(n)$ for $n\leq 2d$ by the $(2d,2)$-Helly's property.
Let us assume $\St(n)$ for some $n\geq 2d$; we want to prove $\St(n+1)$.
Let \ts $f:A\longrightarrow H$, $A\subset \Z^d$, be $1$-Lipschitz with $|A|=n+1$
and \ts $b\in \Z^d\setminus A$. Without loss of generality assume that $b=\m 0$.
Also assume that \ts $\Cull(A, \m 0)=A$; otherwise we can use the induction
hypothesis and Proposition~\ref{prop:culling of extras} to obtain
the required extension to \ts $A\cup \{\m 0\}$.

We will prove that there exists a set \ts $\t A\subset \Z^d$ and a
$1$-Lipschitz function \ts $\t f: \t A\longrightarrow H$, such that
\begin{enumerate}
\item
If $\t f$ has an extension to $\t A\cup\{\m 0\}$ then $f$ has an extension
to \ts $A\cup \{\m 0\}$.
\item
Either the set $\t A$ is contained in the coordinate axes of $\Z^d$ or \ts $|\t A|\leq 2d$.
\end{enumerate}
By Remark~\ref{remark: zd culling coordinates}, if $A$ is contained in the coordinate axis then $|\Cull(A, \m 0)|\leq 2d$. Since $H$ has the $(2d,2)$-Helly's property, by Proposition \ref{prop: small culled vertices} it follows that $\t f$ has an extension to $\t A\cup\{\m 0\}$ which completes the proof.

Since $|A|\geq n+1>2d$, there exists $\mi,\mj\in A$ and a coordinate $1\leq k \leq d$ such that $i_k, j_k$ are non-zero and have the same sign. Suppose $i_k\leq
j_k$. Then there is a geodesic from $\mj$ to $\mi - i_k \m e_k$ which passes through $\mi$. Since $A=\Cull(A, \m 0)$ we have that $$\mi - i_k \m e_k\notin \{\m
0\}\cup A.$$
Thus $\mj \notin \Cull(A, \mi-i_k\m e_k)$ and hence $|\Cull(A, \mi-i_k\m e_k)|\leq n$. By $\St(n)$ there exists a $1$-Lipschitz extension of $f|_{\Cull(A, \mi-i_k\m
e_k)}$ to $\Cull(A, \mi-i_k\m e_k)\cup\{\mi-i_k\m e_k\}$. By Proposition~\ref{prop:culling of extras}
there is a $1$-Lipschitz extension of $f$ to $ f': A\cup
\{\mi-i_k\m e_k\}\longrightarrow H$. But there is a geodesic from $\mi$ to $\m 0$ which passes through $\mi -i_k \m e_k$. Thus
$$\Cull\bigl(A\cup\{\mi-i_k \m e_k\}, \m 0\bigr)\subset \bigl(A\setminus\{\mi\}\bigr)\cup\{\mi -i_k \m e_k\}.$$
Set \ts $A':=\bigl(A\setminus\{\mi\}\bigr)\cup\{\mi -i_k \m e_k\}$. By Proposition~\ref{prop:culling of extras},
map~$f'$ has a $1$-Lipschitz extension to $A\cup \{\mi-i_k\m e_k\}
\cup \{\m 0\}$ if and only if $f'|_{A'}$ has a $1$-Lipschitz extension to $A'\cup \{\m 0\}$.

Thus we have obtained a set $A'$ and a $1$-Lipschitz map $f':A'\subset \Z^d\longrightarrow H$ for which
\begin{enumerate}
\item
If $f'$ has an extension to $A'\cup\{\m 0\}$ then $f$ has an extension to $A\cup \{\m 0\}$.
\item
The sum of the number of non-zero coordinates of elements of $A'$ is less than the sum of the number of non-zero coordinates of elements of $A$.
\end{enumerate}
By repeating this procedure (formally this is another induction) we get the required set $\t A\subset \Z^d$ and $1$-Lipschitz map $\t f: \t A\longrightarrow H$.
This completes the proof.

\bigskip

\section{Applications of the main theorem}\label{section: extensions of main}

\subsection{Back to continuous setting} \label{ss:ext-cont}
The techniques involved in the proof of Theorem~\ref{thm:Zd Kirszbraun}
extend to the continuous case with only minor modifications. The following
result might be of interest in metric geometry.

\smallskip

A metric space $(X,\mt)$ is \emph{geodesically complete} if for all
$x,y\in X$ there exists a continuous function $f: [0,1]\to X$, such that
$$
\mt(x, f(t))\. = \. t \ts \mt(x,y) \quad \text{ and }\quad \mt(f(t), y) \. =\. (1-t)\ts \mt(x,y).
$$

\begin{thm}\label{thm: rd version}
Let $Y$ be a metric space such that every closed ball in $Y$ is compact.
Then $Y$ is $(\R^d, \ell_1)$-Kirszbraun if and only if
$Y$ is geodesically complete and $(2d,2)$-Helly.
\end{thm}

First, we need the following result.

\begin{lemma}\label{lemma: observation}
Let $(X, \mt)$ be separable and the closed balls in $(Y, \mt')$ be compact.
Then $(Y, \mt')$ is $(X, \mt)$-Kirszbraun if and only if all finite sets
$A\subset X$ and $1$-Lipschitz maps $f:A\to Y$, $y\in Y\setminus A$
have a $1$-Lipschitz extension to $A\cup\{y\}$.
\end{lemma}

\begin{proof}
The ``only if'' part is obvious. For the ``if'' part, let
$A\subset X$ and $f:A\to Y$ be $1$-Lipschitz. Since $X$ is separable,
$A$ is also separable. Let
$$
\{x_i~:~i\in \N\} \. \subset \. X \quad \text{ and } \quad \{a_i~:~i\in\N\}\. \subset \. A
$$
 be countable dense sets. By the hypothesis, we have
$$
\bigcap _{1 \leq i\leq  n}B_{\mt(x_1, a_i)}(f(a_i))\, \neq \. \emp.
$$
Since closed balls in $Y$ are compact it follows that
$$
\bigcap_{i\in \N}B_{\mt(x_1, a_i)}(f(a_i))\, \neq \. \emp.
$$
Thus $f$ has a $1$-Lipschitz extension to \ts $A\cup\{x_1\}$.
By induction we get that $f$ has a $1$-Lipschitz extension \ts
$\t f: A\cup\{x_i~:~i\in \N\}\to Y$. Let $g: X\to Y$ be the map given by
$$g(x)\. := \. \lim_{j\to \infty}  \. \t f(x_{i_j}) \quad \text{ for all } \ \, x\in X,
$$
where ${x_{i_j}}$ is some sequence such that $\lim_{j\to \infty} x_{i_j}=x$.
The limit above exists since closed balls in $Y$ are compact, and hence $Y$
is a complete metric space. By the continuity of $\t f$ it follows that
$g|_A=f$ and by the continuity of the distance function it follows that
$g$ is $1$-Lipschitz.
\end{proof}

For the proof of Theorem~\ref{thm: rd version}, note that the main property
of $\Z^d$ exploited in proof of Theorem~\ref{theorem: main result}
is that the graph metric is same as the $\ell_1$ metric.  From the lemma
above, the proof proceeds analogously.  We omit the details.

\smallskip

\subsection{Lipschitz constants} \label{ss:ext-Lip}
The following extension deals with other Lipschitz constants.  In the continuous case this is trivial; however it is more delicate in the discrete setting. Since we are interested in Lipschitz maps between
graphs we restrict our attention to integral Lipschitz constants.

\begin{thm}\label{thm: about t lipschitz maps}
Let $t\in \N$ and $H$ be a connected graph. Then every $t$-Lipschitz map
$f: A\longrightarrow H$, $A\subset \Z^d$, has a $t$-Lipschitz extension to
$\Z^d$ if and only if
$$
\text{for all balls }B_1, B_2, \ldots B_{2d} \text{ of radii multiples
of }t\text{ mutually intersect }\Longrightarrow\cap B_i\neq \emp.
$$
\end{thm}

The proof of Theorem~\ref{thm: about t lipschitz maps} follow verbatim
the proof of Theorem~\ref{thm:Zd Kirszbraun}; we omit the details.

\smallskip

Let $H$ be a $\Z^d$-Kirszbraun graph.  The theorem implies that
all $t$-Lipschitz maps $f: A\longrightarrow H$, $A\subset \Z^d$,
have a $t$-Lipschitz extension. On the other hand, it is easy
to construct graphs $G$ and $H$ for which $H$ is $G$-Kirszbraun
but there exists a $2$-Lipschitz map \ts $f: A\longrightarrow H$,
$A\subset G$, which does not have a $2$-Lipschitz extension.
First, we need the following result.

\begin{prop}\label{prop: small diameter Kirszbraun graph}
Let $ G$ be a finite graph with diameter $n$ and $ H$ be a connected
graph such that $B^H_{n}(v)$ is $G$-Kirszbraun for all $v\in H$.
Then $H$ is a $G$-Kirszbraun graph.
\end{prop}

\begin{proof}
Let $f: A\subset G\longrightarrow H$ be $1$-Lipschitz and pick $a\in A$. Then \ts Image$(f)\subset B^H_n(f(a))$. Since $B^H_{n}(f(a))$ is $G$-Kirszbraun the
result follows.
\end{proof}

Since trees are Helly graphs, we have as an immediate application of the above that $C_{n}$ is $ G$-Kirszbraun if \ts diam$(G)\leq n-1$.
For instance, let $\br = (1,1,1,1,1,1)\in \N^6$ and consider the \emph{star}~$T_{\br}$.  We obtain that $C_6$ is $T_\br$-Kirszbraun.
 Now label the leaves of $T_\br$ as $b_i$, $1\leq i \leq 6$, respectively. For $A=\{b_1, \ldots, b_6\}\subset T_\br$,
 consider the map
$$f: A \to C_6 \ \ \text{ given by } \ \. f(b_i)=i\ts.
$$
The function $f$ is $2$-Lipschitz but it has no $2$-Lipschitz extension to~$T_\br$.

\smallskip

\subsection{Hyperoctahedron graphs} \label{ss:ext-ex}
The \emph{hyperoctahedron graph} $\OO_d$ is the graph obtained by removing a perfect
matching from the complete graph~$K_{2d}$.   Theorem~\ref{theorem: main result}
combined with the following proposition implies that $\OO_d$ are $\Z^{d-1}$-Kirszbraun
but not $\Z^{d}$-Kirszbraun.  When $d=2$ this is the example in the introduction
(see Figure~\ref{f:Z2}).

\begin{prop}\label{p:cross-pol}
Graph $\OO_d$ is $(2d-1,2)$-Helly but not $(2d,2)$-Helly.
\end{prop}

\begin{proof}
Let \ts $B_1, B_2, \ldots, B_{2d-1}$ \ts be balls of radius $\ge 1$. Then, for all
$1\leq i\leq 2d-1$, $B_i\supset \OO_{d}\setminus\{j_i\}$ for some $j_i\in \OO_{d}$. Thus:
$$
\bigcap \ts B_i \, \supseteq \, \OO_d\setminus\{j_i~:~ 1\leq i \leq 2d-1\} \, \neq \, \emp\ts.
$$
This implies that $\OO_d$ is $(2d-1,2)$-Helly.

In the opposite direction, let \ts $B_1, B_1, \ldots, B_{2d}$ \ts
be distinct balls of radius one in~$\OO_d$.  It is easy to see that they intersect pairwise,
but \ts $\cap B_{i}=\emp$.  Thus, graph $\OO_d$ is not $(2d,2)$-Helly, and
Theorem~\ref{thm:Zd Kirszbraun} proves the claim.
\end{proof}

Let us mention that the hyperoctaheron graph \ts $\OO_d\simeq K_{2,\ldots,2}$ \ts
is a well-known obstruction to the Helly's property, see e.g.~\cite{MR2405677}.

\smallskip

\subsection{Bipartite version} \label{ss:ext-bip}
In the study of Helly graphs it is well-known (see e.g.~\cite[$\S$3.2]{MR2405677}) that results which are true with regard to $1$-Lipschitz
extensions usually carry forward to graph homomorphisms in the bipartite case after some small technical modifications. This is also true in our case.

A bipartite graph $H$ is called \emph{bipartite $(n,m)$-Helly} if for balls $B_1, B_2, B_3, \ldots, B_n$ (if $n\neq \infty$ and any finite collection otherwise)
and partite class $H_1$, we have that
any subcollection of size $m$ among $B_1\cap H_1, B_2\cap H_1, \ldots, B_n\cap H_1$ has a nonempty intersection implies
$$\bigcap_{i=1}^n B_i\cap H_1\neq \emp.$$

Let $G, H$ be bipartite graphs with partite classes $G_1, G_2$ and $H_1,H_2$ respectively. The graph $H$ is called \emph{bipartite $G$-Kirszbraun} if for all
$1$-Lipschitz maps $f: A\subset G\longrightarrow H$ for which $f(A\cap G_1)\subset H_1$ and $f(A\cap G_2)\subset H_2$ there exists $\t f\in \Hom(G, H)$ extending it.

\begin{thm}\label{thm: bipartite version of main theorem}
Graph $H$ is bipartite $\Z^d$-Kirszbraun if and only if $H$ is bipartite $(2d,2)$-Helly.
\end{thm}

As noted in the introduction, Theorem~\ref{theorem: main result} implies that graph $\Z^2$ is not $(4,2)$-Helly. However it is
bipartite $(\infty,2)$-Helly (see below).

\smallskip

Given a graph $H$ we say that $v\sim_H w$ to mean that
$(v,w)$ form an edge in the graph. Let $H_1, H_2$ be graphs with vertex sets $V_1, V_2$ respectively.  Define:
\begin{enumerate}
\item
\emph{Strong product} $H_1\boxtimes H_2$ as the graph with the vertex set $V_1\times V_2$, and edges given by
$$\aligned
(v_1,v_2)\, \sim_{H_1\boxtimes H_2}(w_1, w_2) \quad & \text{ if } \ \ v_1=w_1 \ \ \text{and} \ \  v_2\sim_{H_2}w_2\ts, \\
& \text{ or } \ \ v_1\sim_{H_1}w_1 \ \ \text{and} \ \  v_2=w_2\ts, \\
& \text{ or }\ \ v_1\sim_{H_1}w_1 \ \ \text{and} \ \  v_2\sim_{H_2}w_2\ts.
\endaligned
$$
\item
\emph{Tensor product} $H_1\times H_2$ as the graph with the vertex set $V_1\times V_2$, and edges given by
$$
(v_1,v_2)\sim_{H_1\times H_2}(w_1, w_2)\quad \text{ if } \ \ \. v_1\sim_{H_1}w_1\ \ \text{and} \ \  v_2\sim_{H_2}w_2\ts.
$$
\end{enumerate}

\smallskip

\begin{prop}\label{prop: products} If for a graph $G$, graphs $H_1$ and $H_2$ are $G$-Kirszbraun then $H_1\boxtimes H_2$ is $G$-Kirszbraun. If for a bipartite
graph $G$, bipartite graphs
$H_1'$ and $H_2'$ are bipartite $G$-Kirszbraun then the connected components of $H_1'\times H_2'$ are bipartite $G$-Kirszbraun.
\end{prop}

\begin{proof} We will prove this in the non-bipartite case; the bipartite case follows similarly. Let $f:=(f_1, f_2): A\subset G\longrightarrow H_1\boxtimes H_2$
be $1$-Lipschitz. It follows that the functions $f_1$ and $f_2$ are $1$-Lipschitz as well; hence they have $1$-Lipschitz extensions $\t f_1: G\longrightarrow
H_1$ and $\t f_2: G\longrightarrow H_2$. Thus $(\t f_1, \t f_2): G\longrightarrow H_1\boxtimes H_2$ is $1$-Lipschitz and extends $f$.
\end{proof}

\begin{corollary}[{cf.~\cite[$\S$3.2]{MR2405677}}]
Graph $\Z^2$ is bipartite $(\infty,2)$-Helly.
\end{corollary}

\begin{proof}
As we mentioned above, it is easy to see that all trees are Helly graphs.
By Proposition~\ref{prop: products}, so are the connected components of
$\Z\times \Z$ which are graph isomorphic to $\Z^2$.  Now
Theorem~\ref{thm: bipartite version of main theorem} implies the result.
\end{proof}

\bigskip

\section{Complexity aspects}\label{section:recognisability and an application}

\subsection{The recognition problem} Below we  give a polynomial time algorithm
to decide whether a given graph is $\Z^d$-Kirszbraun.  We assume
that the graph is presented by its adjacency matrix.

\begin{prop}\label{prop:recognition}
For all fixed $n,m\in \N$, the recognition problem of
$(n,m)$-Helly graphs and bipartite $(n,m)$-Helly graphs can be
decided in \ts \textrm{poly}$(|H|)$ time.
\end{prop}

For $n=\infty$ and $m=2$, the recognition problem was solved
in~\cite{MR1029165}; that result does not follow from
the proposition.

\begin{proof}
Let us seek the algorithm in the case of $(n,m)$-Helly graphs; as always,
the bipartite case is similar. In the following, for a function $\tg:\R\to \R$ by
$t=O(\tg(|H|))$ we mean $t\leq k \tg(|H|)$, where $k$ is independent
of $|H|$ but might depend on $m,n$.
\begin{enumerate}
\item
Determine the distance between the vertices of the graph. This takes $O(|H|^3)$ time.
\item
Now make a list of all the collections of balls; each collection being of cardinality~$n$.
Since the diameter of the graph $H$ is bounded by $|H|$; listing
the centers and the radii of the balls takes time $O(|H|^{2n})$.
\item
Find the collections for which all the subcollections of cardinality $m$ intersect.
For each collection, this step takes $O(|H|)$ time.
\item
Check if the intersection of the balls in the collections found in the previous
step is nonempty. This step again takes $O(|H|)$ time.
\end{enumerate}
Thus, the total time is $O(|H|^{2n+3})$, as desired.
\end{proof}

\smallskip

\subsection{The hole filling problem} The following application
is motivated by the tileability problems, see Section~\ref{section: motivation}.

Fix $d\geq 2$. By a \emph{box} $B_n$ in $\Z^d$ we mean a subgraph $\{0,1,\ldots,n\}^d$.
By the \emph{boundary} $\partial_n$ we mean the internal vertex boundary of $B_n$,
that is, vertices of $B_n$ where at least one of the coordinates is either $0$ or~$n$.
The \emph{hole-filling problem} asks: \ts Given a graph $H$ and a graph homomorphism \ts
$f\in \Hom(\partial_n, H)$, does it extend to a graph homomorphism \ts $\t f\in \Hom(B_n,H)$?

\begin{thm}\label{prop: hole filling}
Fix $d\ge 1$. Let $H$ be a finite bipartite $(2d,2)$-Helly graph, $B_n\subset \Z^d$
a box and $f\in \Hom(\partial_n, H)$ be a graph homomorphism as above.  Then the
hole-filling problem for $f$ can be decided in \ts \textrm{poly}$(n+|H|)$ time.
\end{thm}

The same result holds in the context of $1$-Lipschitz maps for $(2d,2)$-Helly graphs;
the algorithm is similar. For general~$H$, without the $(2d,2)$-Helly assumption,
the problem is a variation on existence of graph homomorphism, see~\cite[Ch.~5]{HN}.
The latter is famously~\NP-complete in almost all nontrivial cases, which makes the
theorem above even more surprising.

\begin{proof} In the following, for a function $\tg:\R^2\to \R$,  by $t=O(\tg(|H|,n))$
we mean $t\leq k \tg(|H|,n)$ where $k$ is independent of $|H|$ and $n$.
Let $f\in \Hom(\partial_n, H)$ be given. Since $H$ is bipartite $(2d,2)$-Helly graph,
by Theorem~\ref{theorem: main result}, $f$ extends to $B_n$ if and only if
$f$ is $1$-Lipschitz. Thus to decide the hole-filling problem we need to
determine whether or not $f$ is $1$-Lipschitz. This can be decided in
polynomial time:
\begin{enumerate}
\item
Determine the distances between all pairs of vertices in~$H$. This costs $O(|H|^3)$.
\item
For each pair of vertices in the graph $\partial_n$, determine the distance
between the pair and their image under $f$ and verify the Lipschitz condition.
This costs $O(n^{2d-2})$.
\end{enumerate}
The total cost is $O(n^{2d-2}+|H|^3)$, which completes the proof.
\end{proof}

For $d=2$ and $H=\Z$, the above algorithm can be modified to give
a $O(n^2)$ time complexity for the hole-filling problem of an $[n\times n]$
box.  This algorithm can be improved to the nearly linear time
$O(n \log n)$, by using the tools in~\cite{MR3530972}.
We omit the details.

\bigskip

\section{Final remarks and open problems} \label{s:finrem}

\subsection{}
In the view of our motivation, we focus on the $\Z^d$-Kirszbraun property
throughout the paper. It would be interesting to find characterizations
for other bipartite domain graphs such as the hexagonal and the
square-octahedral lattice (cf.~\cite{Ken,Thu}).  Similarly, it would be
interesting to obtain a sharper time bound on the recognition
problem as in Proposition~\ref{prop:recognition}, to obtain applications
similar to~\cite{MR3530972}.

Note that a vast majority of tileability problems are computationally hard,
which makes the search for tractable sets of tiles even more interesting
(see~\cite{Pak-horizons}). The results in this paper, especially the bipartite
versions, give a guideline for such a search.

\subsection{}
As we mention in Section~\ref{section: motivation}, there are
intrinsic curvature properties of the underlying spaces,
which allow for the Helly-type theorems~\cite{langschroder1997}.
In fact, there are more general local hyperbolic properties
which can also be used in this setting, see~\cite{CE,CDV,lang1999}.
See also a curious ``local-to-global'' characterization
of Helly graphs in~\cite{CC+}, and a generalization of
Helly's properties to hypergraphs~\cite{BeD,DPS}.

\subsection{}
In literature, there are other ``discrete Kirszbraun theorems'' stated in different contexts.
For example, papers~\cite{AT,Bre} give a PL-version of the result for finite
$A,B\subset \R^d$ and 1-Lipschitz $f:A\to B$.  Such results are related to the classical
\emph{Margulis napkin problem} and other isometric embedding/immersion problems,
see e.g.~\cite[$\S$38--$\S$40]{Pak} and references therein.

In a different direction, two more related problems are worth mentioning.  
First, the \emph{carpenter's rule problem} can be viewed as a problem of finding 
a ``discrete $1$-Lipschitz homotopy'', see~\cite{CD,MR2007962}.  This is also 
(less directly) related to the well-known \emph{Kneser--Poulsen conjecture}, 
see e.g.~\cite{Bez}.

\bigskip

\subsection*{Acknowledgements} We would like to thank the referee for several useful comments and suggestions. We are grateful to Arseniy Akopyan,
Alexey Glazyrin, Georg Menz, Alejandro Morales and Adam Sheffer
for helpful discussions.  Victor Chepoi kindly read the draft
of the paper and gave us many useful remarks, suggestions and
pointers to the literature.

We thank the organizers of the thematic
school ``Transversal Aspects of Tiling'' at Ol\'{e}ron, France in~2016,
for inviting us and giving us an opportunity to meet and interact
for the first time.  The first author has been funded by ISF grant
Nos.~1599/13, 1289/17 and ERC grant No.~678520.  The second author was
partially supported by the NSF.

\vskip1.1cm




\begin{thebibliography}{ABCDE}

\bibitem[AT08]{AT}
A.~V.~Akopyan and A.~S.~Tarasov,
A constructive proof of Kirszbraun's theorem,
\emph{Math.\ Notes}~\textbf{84} (2008), 725--728.

\bibitem[BC08]{MR2405677}
H.-J.~Bandelt and V.~Chepoi,
Metric graph theory and geometry: a survey,
in {\em Surveys on discrete and computational geometry},
AMS, Providence, RI, 2008, 49--86.

\bibitem[BP89]{MR1029165}
H.-J.~Bandelt and E.~Pesch,
Dismantling absolute retracts of reflexive graphs,
 {\em European J.~Combin.}~\textbf{10} (1989), 211--220.

\bibitem[BL00]{BL}
Y.~Benyamini and J.~Lindenstrauss,
\emph{Geometric nonlinear functional analysis},
Vol.~1, AMS, Providence, RI, 2000.

\bibitem[BD75]{BeD}
C.~Berge and P.~Duchet,
A generalization of Gilmore's theorem, in
\emph{Recent advances in graph theory},
Academia, Prague, 1975,  49--55.

\bibitem[Bez10]{Bez}
K.~Bezdek, \emph{Classical topics in discrete geometry},
Springer, New York, 2010.

\bibitem[BS08]{BS}
A.~I.~Bobenko and Yu.~B.~Suris,
\emph{Discrete differential geometry},
AMS, Providence, RI, 2008.

\bibitem[Bre81]{Bre}
U.~Brehm,
Extensions of distance reducing mappings to piecewise congruent mappings
on~$\R^m$, \emph{J.~Geom.}~\textbf{16} (1981), 187--193.

\bibitem[C+]{CC+}
J.~Chalopin, V.~Chepoi, H.~Hirai and D.~Osajda,
Weakly modular graphs and nonpositive curvature,
to appear in \emph{Memoirs AMS}; {\tt arXiv:1409.3892}

\bibitem[CDV17]{CDV}
V.~Chepoi, F.~F.~Dragan and Y.~Vax\'es,
Core congestion is inherent in hyperbolic networks,
in \emph{Proc.\ 28th SODA}, SIAM, Philadelphia, PA, 2017,
2264--2279.

\bibitem[CE07]{CE}
V.~Chepoi and B.~Estellon,
Packing and covering of $\delta$-hyperbolic spaces by balls,
in \emph{Approximation, Randomization, and Combinatorial Optimization}
(M.~Charikar et al., eds),  Springer, Berlin, 2007, 59--73.

\bibitem[CD04]{CD}
R.~Connelly and E.~D.~Demaine,
Geometry and topology of polygonal linkages, in
\emph{Handbook of Discrete and Computational Geometry} (Second Ed.),
CRC Press, Boca Raton, FL, 2004, 197--218.

\bibitem[CDR03]{MR2007962}
R.~Connelly, E.~D. Demaine and G.~Rote,
Straightening polygonal arcs and convexifying polygonal cycles,
\emph{Discrete Comput.\ Geom.}~\textbf{30} (2003), 205--239.

\bibitem[DGK63]{MR0157289}
L.~Danzer, B.~Gr\"unbaum and V.~Klee,
Helly's theorem and its relatives,
in {\em Proc.\ {S}ympos.\ {P}ure {M}ath.}, vol.~VII, AMS, Providence,
RI, 1963, 101--180.

\bibitem[DPS14]{DPS}
M.~C.~Dourado, F.~Protti and J.~L.~Szwarcfiter,
On Helly hypergraphs with variable intersection sizes,
\emph{Ars Combin.}~\textbf{114} (2014), 185--191.

\bibitem[Fed69]{federer1969}
H.~Federer,
{\em Geometric measure theory},  Springer, New York,  1969.

\bibitem[HN04]{HN}
P.~Hell and J.~Ne\v{s}et\v{r}il,  \emph{Graphs and homomorphisms},
Oxford Univ.\ Press, Oxford, 2004.

\bibitem[Hel23]{helly1923}
E.~Helly, {\"U}ber mengen konvexer k{\"o}rper mit gemeinschaftlichen punkten
(in German), {\em J.~Deutsch.\ Math.\ Vereinig}~\textbf{32} (1923), 175--176.

\bibitem[Ken04]{Ken}
R.~Kenyon,
An introduction to the dimer model,
in \emph{ICTP Lect.\ Notes~XVII}, Trieste, 2004, 267--304.

\bibitem[Kir34]{kirszbraun1934}
M.~Kirszbraun,
 {\"U}ber die zusammenziehende und lipschitzsche transformationen,
 {\em Fund.\ Math.}~\textbf{22} (1934), 77--108.

\bibitem[Lang99]{lang1999}
U.~Lang, Extendability of large-scale Lipschitz maps,
\emph{Trans.\ AMS}~\textbf{351} (1999), 3975--3988.

\bibitem[LS97]{langschroder1997}
U.~Lang and V.~Schroeder,
 Kirszbraun's theorem and metric spaces of bounded curvature,
 {\em Geom.\ Funct.\ Anal.}~\textbf{7} (1997), 535--560.

\bibitem[Lin02]{Lin}
N.~Linial,
Finite metric spaces -- combinatorics, geometry and algorithms,
in \emph{Proc.\ ICM Beijing}, Vol.~III, Higher Ed.\ Press, Beijing, 2002, 573--586.

\bibitem[Mat02]{Mat}
J.~Matou\v{s}ek, \emph{Lectures on discrete geometry},
Springer, New York, 2002.

\bibitem[MT16]{TassyMenz2016}
G.~Menz and M.~Tassy,  A variational principle for a non-integrable model;
{\tt arXiv:1610.08103}.

\bibitem[Pak03]{Pak-horizons}
I.~Pak, Tile invariants: New horizons,
\emph{Theor.\/ Comp.\/ Sci.}~\textbf{303} (2003), 303--331.

\bibitem[Pak09]{Pak}
I.~Pak, \emph{Lectures on discrete and polyhedral geometry},
monograph draft (2009); available at \\
\href{http://www.math.ucla.edu/~pak/book.htm}{\url{http://www.math.ucla.edu/~pak/book.htm}}

\bibitem[PST16]{MR3530972}
I.~Pak, A.~Sheffer and M.~Tassy,
Fast domino tileability,
{\em Discrete Comput.\ Geom.}~\textbf{56} (2016), 377--394.

\bibitem[She05]{MR2251117}
S.~Sheffield, Random surfaces, {\em Ast\'erisque}~\textbf{304} (2005), 175~pp.

\bibitem[Tas14]{Tas}
M.~Tassy, \emph{Tiling by bars}, Ph.D.~thesis, Brown University, 2014; available at \\
\href{http://www.math.ucla.edu/~mtassy/articles/PhD_thesis.pdf}{\url{http://www.math.ucla.edu/~mtassy/articles/PhD_thesis.pdf}}

\bibitem[Thu90]{Thu}
W.~P.~Thurston,
Groups, tilings and finite state automata,
in \emph{Lecture Notes}, AMS Summer Meetings, Bolder, CO,
1989.

\bibitem[Val45]{MR0011702}
F.~A.~Valentine,
 A {L}ipschitz condition preserving extension for a vector function,
 {\em Amer.\ J.~Math.}~\textbf{67} (1945), 83--93.

\vskip1.0cm

\end{thebibliography}
\end{document}